\documentclass[reqno]{amsart}
\usepackage{amsmath,amssymb,mathrsfs}

\newtheorem{thm}{Theorem}[section]

\newtheorem{lem}[thm]{Lemma}
\newtheorem{prop}[thm]{Proposition}

\newtheorem{rmk}[thm]{Remark}

\newtheorem*{thma}{Theorem A}
\newtheorem*{thmb}{Theorem B}
\newtheorem*{prc}{Proposition C}
\newtheorem*{prd}{Proposition D}

\numberwithin{equation}{section}

\begin{document}

\title[On the flips for a synchronized system]{On the flips for a synchronized system}

\author[H. Choi and Y.-O. Kim]{Hyekyoung Choi and Young-One Kim}

\thanks{The first author was supported by the BK21-Mathematical Sciences Division.}
\address{Hyekyoung Choi: Departmemt of Mathematical Sciences and Research Institute of Mathematics,
Seoul National University, Seoul 151-747, Korea}\email{yuritoy1@snu.ac.kr}
\address{Young-one Kim: Departmemt of Mathematical Sciences and Research Institute of Mathematics,
Seoul National University, Seoul 151-747, Korea}\email{kimyo@math.snu.ac.kr}

\maketitle
\begin{abstract}
It is shown that if an infinite synchronized system has a flip, then
it has infinitely many non-conjugate flips, and that the result
cannot be extended to the class of coded systems.
\end{abstract}

\section{Introduction}
Let $X$ be a (two-sided) shift space, so that $(X,\sigma_{X})$  is
an invertible topological dynamical
    system. A homeomorphism $\varphi : X \rightarrow X$ is called a \textit{flip}
    for $(X, \sigma_{X})$ if $\varphi^2 =\text{id}_{X}$ and
    $\varphi\sigma_{X}=\sigma_{X}^{-1}\varphi$. In this case, we say
    that $(X, \sigma_{X}, \varphi)$ is a \textit{shift-flip system}.
The simplest one may be the reversal map $\rho$ defined by
$\rho(x)_i = x_{-i}$, provided $X$ is closed under $\rho$.
    Two shift-flip systems
    $(X, \sigma_{X}, \varphi)$ and $(Y, \sigma_{Y}, \psi)$ are said to be \textit{conjugate} if there is a homeomorphism
    $\Phi : X \rightarrow Y$ such that $\Phi\circ\sigma_{X}=\sigma_{Y}\circ \Phi$ and $\Phi\circ\varphi=\psi\circ
    \Phi$. The homeomorphism $\Phi$ is called a {\it conjugacy} from  $(X, \sigma_{X}, \varphi)$ to $(Y, \sigma_{Y}, \psi)$.
    Two flips $\varphi$ and $\varphi'$ for $(X,\sigma_{X})$ are said to be
    conjugate if the shift-flip systems $(X, \sigma_{X}, \varphi)$
    and $(X, \sigma_{X}, \varphi')$ are conjugate. Thus the flips $\varphi$ and $\varphi'$
    are conjugate if and only if there is an automorphism $\theta$ of $(X,\sigma_{X})$ such
    that $\theta\varphi=\varphi' \theta$.

Suppose $\varphi$ is a flip for $(X,\sigma_{X})$. Then the shift
dynamical systems $(X,\sigma_{X})$ and $(X,\sigma_{X}^{-1})$ are
conjugate, but this does not hold in the general case (\cite{LM},
Examples 7.4.19 and 12.3.2). Thus not every shift dynamical system
has a flip. On the other hand, it is easy to see that the maps
$\sigma_{X}^m\varphi$, $m\in\Bbb Z$, are flips for $(X,\sigma_{X})$,
and that they are all distinct whenever $|X|=\infty$. Therefore if
an infinite shift dynamical system has a flip, then it has
infinitely many different ones. But $\sigma_{X}^m\varphi$ and
$\sigma_{X}^n\varphi$ are conjugate whenever $n-m$ is even, because
$$
\sigma_{X}^{(n-m)/2}\left(\sigma_{X}^m\varphi\right) =
\left(\sigma_{X}^n\varphi\right)\sigma_{X}^{(n-m)/2}.
$$
Although the flips $\sigma_{X}^m\varphi$, $m\in\Bbb Z$, are all
distinct, each of them is conjugate to one of the two flips
$\varphi$ and $\sigma_X\varphi$.

In this paper, we will be interested in the property that {\it if
$(X, \sigma_X)$ has a flip, then it has infinitely many
non-conjugate ones.} In \cite{KLP}, it is shown that if
$|\mathcal{A}|\geq 2$, then the full $\mathcal A$-shift
$(\mathcal{A}^{\mathbb{Z}}, \sigma)$ has infinitely many
non-conjugate flips. Let $M$ denote the Morse shift \cite{GH, MH},
that is, $M$ is the set of bi-infinite sequences in $\{0,1\}^{\Bbb
Z}$ which does not contain any block in the set
$$
\{awawa : a\in\{0,1\}\text{ and } w\in\mathcal{B}(\{0,1\}^{\Bbb
Z})\}.
$$
It is evident that the reversal map $\rho$ and the map $\psi : M\to
M$ defined by  $\psi(x)_i = 1-x_{-i}$ are flips for $(M, \sigma_M)$.
Suppose $\varphi$ is a flip for $(M, \sigma_M)$. Then $\rho\varphi$
is an automorphism of $(M, \sigma_M)$. But it is known \cite{Cov}
that the automorphism group of $(M, \sigma_M)$ is generated by
$\{\sigma_M, \rho\psi\}$, and we have $(\rho\psi)^2=\text{id}_M$;
hence $\varphi$ is conjugate to one of the four flips $\rho$,
$\psi$, $\sigma_M\rho$ and $\sigma_M\psi$. Thus every infinite full
shift has the property, while the Morse shift does not.

The purpose of this paper is to establish the following:

\begin{thma}\label{thmA}
    If $X$ is a synchronized system, $|X|=\infty$ and there is a flip for $(X,\sigma_{X})$,
     then there are infinitely many non-conjugate flips for $(X,\sigma_{X})$.
\end{thma}

\begin{thmb}\label{thmB}
    There is an infinite coded system $W$ such that the reversal map $\rho$
    is a flip for $(W, \sigma_W)$
    and $\{ \sigma_W^m\rho : m\in\Bbb Z\}$ is the set of flips for $(W,
    \sigma_W)$.
\end{thmb}

In Theorem B, the result implies that every flip for $(W, \sigma_W)$
is conjugate to one of the two flips $\rho$ and $\sigma_W\rho$,
because $\sigma_W^m\rho$ and $\sigma_W^n\rho$ are conjugate whenever
$m-n$ is even. We recall that every irreducible sofic shift is a
synchronized system and every synchronized system is coded
\cite{Kri, LM}.

If $\varphi$ is a flip for $(X, \sigma_X)$, we denote the set $\{
x\in X : \sigma_X^n(x)=\varphi(x)=x\}$ by $F(\varphi; n)$ for $n=1,
2, 3, \dots$, and denote by $A(\varphi)$ the  set of points $x\in X$
such that a finitary (intrinsically synchronizing) block appears
infinitely often in $x$ and
$$
0<\left|\{i\in\Bbb Z : \varphi(x)_i \ne x_i\}\right|<\infty.
$$
It is clear that if $(X, \sigma_{X}, \varphi)$ and $(Y, \sigma_{Y},
\psi)$ are shift-flip systems and $\Phi$ is a conjugacy from  $(X,
\sigma_{X}, \varphi)$ to $(Y, \sigma_{Y}, \psi)$, then
$\Phi(F(\varphi; n))= F(\psi; n)$ for all $n$, and
$\Phi(A(\varphi))=A(\psi)$.  It is also clear that $|F(\varphi;
n)|<\infty$ for all $n$. Now, Theorem A is an immediate consequence
of the following technical results which are proved in the next
section, and an inductive argument.

\begin{prc}
    If $X$ is an infinite synchronized system and $\varphi$ is a flip for $(X, \sigma_{X})$,
    then at least one of the two sets $A(\varphi)$ and $A(\sigma_X \varphi)$ is non-empty.
    \end{prc}

 \begin{prd}
    If $X$ is an infinite synchronized system, $\varphi$ is a flip for $(X,
    \sigma_{X})$, and $A(\varphi) \ne \emptyset$, then there is a flip $\psi$ for $(X, \sigma_{X})$ such that
 \begin{itemize}
      \item[(a)] $|F(\varphi; n)| \leq |F(\psi; n)|$ for all $n$,
      \item[(b)] $|F(\varphi; n)| < |F(\psi; n)|$ for some $n$, and
      \item[(c)] $A(\psi) \ne \emptyset$.
     \end{itemize}
    \end{prd}

We prove Theorem B by constructing $W$ explicitly (Section 3). In
\cite{FF}, it is shown that there is a coded system whose
automorphism group is generated by the shift map and is isomorphic
to $\Bbb Z$ (\cite{FF}, Corollary 2.2), but it is not clear whether
the coded system has a flip. Fortunately, we can simplify the
construction to obtain a coded system $W$ such that the automorphism
group of $(W, \sigma_W)$ is $\{\sigma_W^m : m\in\Bbb Z\}$ and the
the reversal map $\rho$ is a flip for $(W, \sigma_W)$. It is then
trivial to see that $W$ has the property stated in Theorem B.

\section{Proof of Propositions C and D}

We start with some preliminaries. Let $X$ be a shift space and
$\varphi$ be a flip for $(X, \sigma_{X})$.
    If $\varphi(x)_{0}=\varphi(x')_{0}$ whenever $x_{0}=x'_{0}$,
    then there is a unique map $\tau: \mathcal{B}_{1}(X) \rightarrow \mathcal{B}_{1}(X)$ such that
    \begin{displaymath}
     \varphi(x)_{i}=\tau(x_{-i}) \quad (x\in X , i \in \mathbb{Z}),
    \end{displaymath}
    and consequently $\tau^2=\text{id}_{\mathcal{B}_{1}(X)}$. In this case, we say that $\varphi$ is a
    \textit{one-block} flip  and $\tau$ is the \textit{symbol map} of $\varphi$. The following lemma states that every flip for a
    shift dynamical system can be recoded to a one-block flip.

    \begin{lem}\label{lem1}
    Suppose $X$ is a shift space and $\varphi$ is a flip for $(X,
    \sigma_{X})$. Then there are a finite set $\mathcal{A}$, a shift space
    $Y$ over $\mathcal{A}$, and a one-block flip $\psi$
    for $(Y, \sigma_{Y})$ such that $(Y, \sigma_{Y}, \psi)$ is conjugate to $(X, \sigma_{X},
    \varphi)$.
    \end{lem}

    \begin{proof} Let $\mathcal{A} = \{ (x_0, \varphi(x)_0) : x\in X \}$. For
    $x=\langle x_i \rangle_{i\in \Bbb Z}\in X$ let $\Phi(x)$ denote the bi-infinite sequence
    $\left\langle(x_i, \varphi(x)_{-i})\right\rangle_{i\in\Bbb Z}$ in $\mathcal{A}^{\Bbb Z}$, and set
    $Y=\{\Phi(x) : x\in X\}$. It is then clear that $Y$ is a shift
    space and $\Phi$ is a conjugacy between the shift spaces $X$ and $Y$. If we put $\psi=\Phi\circ\varphi\circ\Phi^{-1}$,
     then $\psi$ is a one-block flip for $(Y, \sigma_{Y})$, $\Phi$ is a
    conjugacy from $(X, \sigma_{X},\varphi)$ to $(Y, \sigma_{Y},\psi)$,  and the symbol map is given
    by $\mathcal{A}\ni (a,b)\mapsto (b,a) \in\mathcal{A}$.
    \end{proof}

Suppose $\varphi$ is a one-block flip and $\tau$ is the symbol map.
For notational simplicity, we write $\tau(a) = a^*$ for
$a\in\mathcal B_1(X)$, and if $w=w_1 w_2 \cdots w_n \in \mathcal
B_n(X)$, we denote the block $w_n^*\cdots w_2^* w_1^*$ by $w^*$.
Thus we have $\varphi(x)_{[i, j]}=\left(x_{[-j, -i]}\right)^*$ for
$x\in X$ and for $i\leq j$. It is clear that $w^*\in\mathcal B(x)$
whenever $w\in \mathcal B(X)$, and that $(w^*)^*=w$. It is also
clear that if $w$ is finitary, then so is $w^*$. If $N$ is a
positive integer, the map $\varphi^{[N]} : X^{[N]} \to X^{[N]}$
defined by
$$
\varphi^{[N]}(y)_i = (y_{-i})^*
$$
is a one-block flip for the $N$-th higher block system $(X^{[N]},
\sigma_{X^{[N]}})$ of $(X, \sigma_X)$. It is easy to see that if $N$
is odd, then $(X^{[N]}, \sigma_{X^{[N]}}, \varphi^{[N]})$ is
conjugate to $(X, \sigma_X, \varphi)$, otherwise it is conjugate to
$(X, \sigma_X, \sigma_X\varphi)$.

In our proof of the propositions, the following lemma will play a
crucial role.

\begin{lem}
   Suppose that $X$ is an irreducible shift space, $|X|=\infty$, and $f\in\mathcal B_1(X)$. Then there are
   blocks $a, b \in \mathcal B(X)$ such that
   \begin{itemize}
      \item[(a)] $faf, fbf \in \mathcal B(X)$,
      \item[(b)] $f$ does not appear in $b$, and
      \item[(c)] $fbfa$ does not appear in any $(|a|+1)$-periodic point.
     \end{itemize}
    \end{lem}

 \begin{proof} Since $X$ is irreducible and $|X|=\infty$, there are
 blocks $a$ and $b$ such that $faf, fbf \in \mathcal B(X)$ and the
 following hold:

 \begin{itemize}
      \item[(i)] if $fa'f \in \mathcal B(X)$, then $|a|\leq |a'|$,
      \item[(ii)] $b\ne (af)^n a$ for all $n\geq 0$, and
      \item[(iii)] if $fb'f \in \mathcal B(X)$ and $b'\ne (af)^n a$ for all $n\geq
 0$, then $|b|\leq |b'|$.
     \end{itemize}
It is then easy to see that the blocks $a$ and $b$ have the desired
properties.
    \end{proof}

We recall that a \textit{synchronized system} is an irreducible
shift space which has a finitary block.

\begin{lem}
   Suppose that $X$ is an infinite synchronized system and $\varphi$ is
   a flip for $(X, \sigma_X)$. If there is a point $x\in X$ such that $\varphi(x)=x$ and a
   finitary block appears in $x$, then $A(\varphi)\ne\emptyset$.
    \end{lem}

 \begin{proof}
Suppose $x\in X$, $\varphi(x)=x$ and a finitary block appears in
$x$. By Lemma 2.1, we may assume that $\varphi$ is a one-block flip.
 Then $(x_{[-n, n]})^*=x_{[-n, n]}$
for all $n\geq 0$. Since a finitary block appears in $x$, it follows
that $x_{[-n, n]}$ is also finitary whenever $n$ becomes
sufficiently large. Suppose $x_{[-n, n]}$ is finitary. By passing to
the $(2n+1)$-st higher block system, we may assume that there is a
finitary symbol (finitary block of length one) $f$ such that
$f^*=f$. Let $a, b\in\mathcal B(X)$ satisfy the conditions in Lemma
2.2, and let $N$ be a positive integer such that
\begin{equation}
2|a|+1+ 2(|b|+1)\leq (N-1)(|a|+1).
\end{equation}
If we write $afa (fb)^2 (fa)^{2N} =w$ and $|a|+|b|+1+N(|a|+1)=M$,
then $|w|=2M+1$ and $(a^*f)^k w (fa)^k \in \mathcal B(X)$ for all
$k\geq 0$. Hence there is a point $y\in X$ such that $y_{(-\infty,
-M-1]}=\cdots a^*fa^*fa^*f$, $y_{[-M, M]}=w$ and $y_{[M+1,
\infty)}=fafafa\cdots$. It is evident that the finitary symbol $f$
appears infinitely often in $y$, and $\varphi(y)_i=y_i$ whenever
$|i|\geq M+1$. Since the block $fbfa$ does not appear in any
$(|a|+1)$-periodic point, (2.1) implies that $w^* \ne w$; hence
$\varphi(y)_i \ne y_i$ for some $i\in [-M, M]$. This proves the
lemma.       \end{proof}

\begin{proof}[Proof of Proposition C]
Suppose $X$ is an infinite synchronized system, $\varphi$ is a flip
for $(X, \sigma_X)$ and $f$ is a finitary block. By Lemma 2.1, we
may assume that $\varphi$ is a one-block flip. Then $f^*$ is also a
finitary block, and there are blocks $v$ and $w$ such that $fvf, f^*
w f \in \mathcal B(X)$.

We first consider the case when $|w|$ is odd. If we write $|w|=2N
+1$, then there is a point $x\in X$ such that $x_{(-\infty, -N-1]} =
\cdots v^* f^*v^* f^*v^* f^*$, $x_{[-N, N]}=w$ and $x_{[N+1,
\infty)}=fvfvfv\cdots$. It is obvious that the finitary block $f$
appears in $x$ infinitely often, $\varphi(x)_i = x_i$ whenever
$|i|\geq N+1$ and that $\varphi(x)_{[-N, N]}=w^*$. If $w^* = w$,
then $\varphi(x)=x$, and Lemma 2.3 implies that
$A(\varphi)\ne\emptyset$; otherwise we have $\varphi(x)\ne x$, and
hence $x\in A(\varphi)$.

Now, suppose that $N$ is a positive integer and $|w|=2N$. Then there
is a point $x\in X$ such that $x_{(-\infty, -N-1]} = \cdots v^*
f^*v^* f^*v^* f^*$, $x_{[-N, N-1]}=w$ and $x_{[N,
\infty)}=fvfvfv\cdots$. In this case, we have $\sigma_X\varphi(x)_i
= x_i$ for all $i\in (-\infty, -N-1] \cup [N, \infty)$, and
$\sigma_X \varphi(x)=x$ if and only if $w^* = w$. Thus we have
$A(\sigma_X \varphi)\ne\emptyset$, by the same reasoning as in the
first case.
\end{proof}

For $A, B \subset \Bbb Z$ and $m\in\Bbb Z$, we denote the sets $\{mj
: j\in A\}$ and $\{ j+k : j\in A \text{ and } k\in B\}$ by $mA$ and
$A+B$, respectively. We also write $(-1)A=-A$ and $\{ m\} + A =
m+A$.

\begin{proof}[Proof of Proposition D]
Suppose $X$ is an infinite synchronized system, $\varphi$ is a flip
for $(X, \sigma_X)$ and $A(\varphi)\ne \emptyset$. We prove the
proposition by constructing an automorphism $\theta$ of $(X,
\sigma_X)$ and homeomorphisms $\theta_1, \theta_2, \theta_3, \dots$
from $X$ onto itself such that $\theta\varphi$ is a flip for $(X,
\sigma_X)$, and the following hold:
\begin{itemize}
         \item[(i)] $\theta_n(F(\varphi; n)) \subset  F(\theta\varphi;
      n)$ for all $n$,
      \item[(ii)] $\theta_n(F(\varphi; n)) \ne F(\theta\varphi;
      n)$ for some $n$, and
      \item[(iii)] $A(\theta\varphi)\ne\emptyset$.
     \end{itemize}

By Lemma 2.1, we may assume that $\varphi$ is a one-block flip. Let
$f$ be a finitary block. We may assume that $|f|$ is odd. By passing
to the $|f|$-th higher block system, we may assume that $f$ is a
finitary symbol. Let $a, b\in \mathcal B(X)$ satisfy the conditions
in Lemma 2.2. Since $X$ is irreducible and $A(\varphi)\ne
\emptyset$, there is a block $c$ of odd length such that $f^* c
f\in\mathcal B(X)$ and $c^*\ne c$. Let $N$ be a positive integer
such that
\begin{equation}
|a|+ 2|b|+ |c| + 2\leq (N-1)(|a|+1),
\end{equation}
and put $fb(fa)^N=d$. Then $d$ and $d^*$ are finitary blocks, and
$d^*cd, d^*c^*d\in\mathcal B(X)$. For notational simplicity, we
write $|c|=2\alpha +1$ and $\alpha + |d|=\beta$, so that $|d^* c d|=
|d^* c^* d| = 2\beta + 1$. Since the symbol $f$ does not appear in
the block $b$, and since the block $fbfa$ does not appear in any
$(|a|+1)$-periodic point, the inequality (2.2) implies that the
following statement holds: If $x\in X$, $i\ne j$, and $x_{[i-\beta,
i+\beta]}, x_{[j-\beta, j+\beta]} \in \{d^*cd, d^*c^*d\}$, then
$|i-j|\geq |c| + |d| + 1$. For $x\in X$ let $\mathcal M(x)$ denote
the set of integers $i$ such that $x_{[i-\beta, i+\beta]} \in
\{d^*cd, d^*c^*d\}$. Then we have
\begin{equation}
[i-\alpha - 1, i + \alpha + 1] \cap [j-\beta, j+\beta] = \emptyset
\qquad (i, j\in\mathcal M(x),\ i\ne j)
\end{equation}
for every $x\in X$. In particular, the intervals $[i-\alpha,
i+\alpha]$, $i\in \mathcal M(x)$, are mutually disjoint. For each
$i\in \mathcal M(x) + [-\alpha, \alpha]$ let $c(i; x)$ denote the
unique element of $\mathcal M(x)$ such that $i\in [c(i; x) - \alpha,
c(i; x) + \alpha]$.

For $A\subset \Bbb Z$ and $x\in X$ we define the bi-infinite
sequence $\theta_A(x)$ by
$$
\theta_A(x)_i=\begin{cases}(x_{2c(i;x)-i})^* \  \ \text{ if }\ i\in
\mathcal (\mathcal M(x)\cap A) + [-\alpha, \alpha],
\\ x_i  \qquad \quad \quad\ \ \, \, \, \text{otherwise}. \end{cases}
$$
Thus $\theta_A$ replaces the part $x_{[i-\alpha, i+\alpha]}$ of $x$
with $(x_{[i-\alpha, i+\alpha]})^*$ whenever $i\in\mathcal M(x)\cap
A$ and leaves the remaining part of $x$ unchanged. Since $d$ and
$d^*$ are finitary blocks and $d^*cd, d^*c^*d\in\mathcal B(X)$, we
have $\theta_A(x)\in X$ for all $x\in X$. From this and (2.3), it
follows that $\mathcal M(x) = \mathcal M(\theta_A(x))$, and
consequently $\theta_A(\theta_A(x))=x$ for all $x\in X$. If $x,
x'\in X$, $i\in \Bbb Z$ and $x_{[i-\alpha -\beta, i+\alpha +
\beta]}=x'_{[i-\alpha -\beta, i+\alpha + \beta]}$, then
$\theta_A(x)_i = \theta_A(x')_i$. Hence $\theta_A : X \to X$ is a
homeomorphism satisfying $\theta_A^2 = \text{id}_X$ for every
$A\subset \Bbb Z$. It is easy to see that $\theta_A\theta_B =
\theta_{(A\triangle B)}$, $\sigma_X \theta_A =
\theta_{(-1+A)}\sigma_X$ and $\varphi\theta_A =
\theta_{(-A)}\varphi$.

We define $\theta$ to be $\theta_{\Bbb Z}$. It is clear that
$\theta$ is an automorphism of $(X, \sigma_X)$ such that $\theta^2 =
\text{id}_X$ and $\theta\varphi=\varphi\theta$. In particular, the
map $\theta\varphi$ is a flip for $(X, \sigma_X)$. For $n=1, 2, 3,
\dots$ we set
$$
H(n)=\bigcup_{k\in\Bbb Z}\left \{ i\in\Bbb Z :
nk<i<n\left(k+\frac{1}{2}\right)\right\},
$$
and define $\theta_n$ to be $\theta_{H(n)}$.

To prove (i), suppose $n$ is a positive integer and $x\in F(\varphi;
n)$, that is, $\sigma_X^n(x)=\varphi(x)=x$. We have $n+ H(n) =
H(n)$, hence $\sigma_X^n(\theta_n(x))=\theta_n(x)$. If we put
$C=\Bbb Z \setminus (H(n)\cup (-H(n)))$, then $\{ H(n), -H(n), C\}$
is a partition of $\Bbb Z$. We have $C=n\Bbb Z$ in the case when $n$
is odd, and $C=n\Bbb Z + \{0, n/2\}$ in the case when $n$ is even.
Since $\sigma_X^n(x)=\varphi(x)=x$ and $c^* \ne c$, it follows that
$\mathcal M(x) \cap C = \emptyset$. Hence we have
$$
\mathcal M(x) \cap (\Bbb Z \triangle (-H(n))) = \mathcal M(x) \cap
(H(n)\cup C) = \mathcal M(x)\cap H(n),
$$
so that
$$
\theta\varphi(\theta_n(x)) = \theta \theta_{(-H(n))} \varphi (x) =
\theta_{(\Bbb Z \triangle (-H(n)))}(x) = \theta_n(x).
$$
Thus $\theta_n(x) \in F(\theta\varphi ; n)$, and (i) is proved.

To prove (ii) and (iii), we first construct a point in $X$. Since
$X$ is irreducible, there is a block $w$ such that $dwd^* \in
\mathcal B(X)$. If we write $2(|c| + 2|d| + |w|)=n$ and $ |c| + 2|d|
+ n = 2m+1$, then we have $n=|w^* d^* cd w d^* c d|$, and there is a
point $z\in X$ such that $\sigma_X^n(z)=z$ and
$$
z_{[-m, m]}=d^* c d w^* d^* cd w d^* c d.
$$
We have $z_{[-\alpha, \alpha]} = c$, and this implies that
$\varphi(z) \ne z$, because $c^* \ne c$. Hence $z\notin F(\varphi;
n)$. Since $\sigma_X^n(z)=z$, we have
$\sigma_X^n(\theta_n(z))=\theta_n(z)$. Let $C$ be as in the proof of
(i). Then $C=n\Bbb Z + \{0, n/2\}$ and we have
$\varphi(z)=\theta_C(z)$, hence
$$
\theta\varphi(\theta_n(z)) = \theta \theta_{(-H(n))} \varphi (z)
=\theta \theta_{(-H(n))} \theta_C (z) = \theta_n(z).
$$
Thus $\theta_n(z)\in F(\theta\varphi ; n)$, while $z\notin
F(\varphi; n)$. Since $\theta_n$ is one-to-one, we see that (ii)
holds. Finally, we have $\theta\varphi(\theta_n(z)) = \theta_n(z)$,
and the finitary block $d$ appears in $\theta_n(z)$, hence (iii)
follows from Lemma 2.3. This proves the proposition.
\end{proof}

\section{Proof of Theorem B}

We recall that a {\it coded system} $X$ is a shift space which has a
code, that is, a set $\mathcal C$ of blocks such that the set of
bi-infinite concatenations of blocks from $\mathcal C$ is dense in
$X$.  It is clear that for every set $\mathcal C$ of blocks over a
finite alphabet there is a unique coded system for which $\mathcal
C$ is a code. In this section, we follow the method given in Section
1 of \cite{FF} to construct an infinite coded system $W$ such that
$W$ is closed under the reversal map $\rho$ and its automorphism
group is $\{\sigma_W^m : m \in \Bbb Z\}$.

Let $\mathcal A = \{0, 1, 2\}$, $I=\bigcup_{k\geq 1} [2^{2k}, 2^{2k
+1}] = [4, 8] \cup [16, 32] \cup \cdots$, and $J=\{(0,0), (1,1),
(1,2), (2,1), (2,2)\}$. In \cite{FF}, a block $w\in \mathcal
B(\mathcal A^{\Bbb Z})$ is defined to be {\it stable} if
\begin{itemize}
         \item[(1)] the blocks $12$ and $21$ do not appear in $w$, and
      \item[(2)] if $x\in\mathcal A^{\Bbb Z}$, $x_{[1, 3|w|+2]}=0^{|w|+1} w 0^{|w|+1}$,
      $1\leq n \leq |w|$, $a\in\{1,2\}$, $|w|+1\leq i < j\leq 2|w|+2$, and $x_{[i,j]}=0a^n0$, then
      $a=1$ if and only if $(n, (x_{i-n}, x_{j+n})) \in (I\times J) \cup (I^C \times J^C)$.
     \end{itemize}
We set $\mathcal C = \{0\} \cup \{0^{|w|+1} w 0^{|w|+1} : w \text{
is stable} \}$, and define $W$ to be the coded system for which
$\mathcal C$ is a code. The following are easy consequences of the
definitions:
\begin{itemize}
         \item[(a)] For each integer $j\ne 0$ the sets $\{ n : n,
         n+j \in I\}$, $\{ n : n\in I \text{ and }  n+j \notin I\}$,
         $\{ n : n\notin I \text{ and }  n+j \in I\}$ and $\{ n : n,
         n+j \notin I\}$ are all infinite.
      \item[(b)] $0^n$ is stable for all $n$, $1^n$ is stable if and
      only if $n\in I$, and $2^n$ is stable if and only if $n\notin
      I$.
      \item[(c)] If $w=w_1 w_2 \dots w_{|w|}$ is stable, then so is
      the reversed block $w_{|w|} \dots  w_2 w_1$.
      \item[(d)] If $w$ is stable, then $0^{\infty}.w 0^{\infty}\in
      W$.
      \item[(e)] If $w$, $w'$ are stable and $n\geq\max\{|w|,
      |w'|\}$, then $w 0^{n+1} w'$ is stable. In particular, every
      finite concatenation of blocks from $\mathcal C$ is stable.
      \item[(f)] $x\in W$ if and only if for every $n\geq 0$ there
      is a stable block $w$ such that $x_{[-n, n]}$ is a subblock of
      $w$.
      \item[(g)] Suppose $x\in W$. Then the blocks $12$ and $21$ do
      not appear in $x$. If $i<j$, $a\in\{1,2\}$, $n>0$ and $x_{[i,j]}=0a^n0$, then
      $a=1$ if and only if $(n, (x_{i-n}, x_{j+n})) \in (I\times J) \cup (I^C \times J^C)$.
     \end{itemize}
Let $W_0$ denote the set of $x\in W$ such that $|\{i\in\Bbb Z : x_i
\ne 0\}|<\infty$. For $x\in W$ we set $\mathcal Z(x) = \{i\in\Bbb Z
: x_i = 0\}$. By (d), (f) and (g), we have the following:
\begin{itemize}
      \item[(h)] $W_0$ is a dense subset of $W$. If $x, x' \in W_0$
      and $\mathcal Z(x) = \mathcal Z(x')$, then $x=x'$.
   \end{itemize}

   First of all, it is obvious that $|W|=\infty$. By (c) and (f),
   $W$ is closed under the reversal map $\rho$. If $\varphi$ is a
   flip for $(W, \sigma_W)$, then $\rho\varphi$ is an automorphism
   of $(W, \sigma_W)$. Hence Theorem B is an immediate consequence
   of the following.

\begin{prop}
The automorphism group of $(W, \sigma_W)$ is $\{\sigma_W^m : m \in
\Bbb Z\}$.
\end{prop}

\begin{rmk}
{\rm This proposition is a simplified version of Proposition 1.6 in
\cite{FF}, and is proved by essentially the same reasoning as in
that paper. Nevertheless, we present the proof here for the readers
convenience.}
\end{rmk}

\begin{proof}
By (h), it is enough to show that for every automorphism $\theta$ of
$(W, \sigma_W)$ there is an integer $m$ such that
\begin{equation}
\mathcal Z(\sigma_W^m \theta (x)) = \mathcal Z (x) \qquad (x\in
W_0).
\end{equation}
Suppose  $\theta$ is an automorphism of $(W, \sigma_W)$. Then there
is a non-negative integer $N$ and there is a block map $\Theta :
\mathcal B_{2N+1}(W) \to \mathcal B_1(W)$ such that
\begin{equation}
\theta(x)_i = \Theta(x_{[i-N, i+N]}) \qquad (x\in W,\  i\in\Bbb Z).
\end{equation}
It is clear that $0^{\infty}, 1^{\infty}, 2^{\infty} \in W$ and
$\theta\left( \{0^{\infty}, 1^{\infty}, 2^{\infty}\}\right) =
\{0^{\infty}, 1^{\infty}, 2^{\infty}\}$.

We first show that $\theta(0^{\infty})=0^{\infty}$. Suppose, to
obtain a contradiction, that $\theta(0^{\infty})\ne 0^{\infty}$.
Then there are $a, b \in \{1, 2\}$ such that
$\theta(a^{\infty})=0^{\infty}$ and $\theta(0^{\infty})=b^{\infty}$.
Let $M$ be a positive integer such that $a^M$ is stable and $M\geq
2N + 1$. By (d) and (e), we have
$$
0^{\infty}.a^M 0^n a^M 0^{\infty} \in W
$$
whenever $n\geq M+1$. Since the blocks $12$ and $21$ do not appear
in any $x\in W$, $\theta(a^{\infty})=0^{\infty}$,
$\theta(0^{\infty})=b^{\infty}$, and since $M\geq 2N + 1$, (3.2)
implies that there are blocks $u$, $v$ and integers $j$, $k$ such
that $j\geq -2N$ and
$$
\sigma_W^k \theta (0^{\infty}.a^M 0^n a^M
0^{\infty})=b^{\infty}u.0b^{n+j} 0v b^{\infty}
$$
for all $n\geq M+1$. If we denote the right hand side by $y$, then
we have $y_{[0, n+j+1]}=0b^{n+j} 0$. Suppose $n+j>\max\{|u|, |v|\}$.
Then $(y_{-(n+j)}, y_{2(n+j)+1})=(b,b)\in J$, hence (g) implies that
$b=1$ in the case when $n+j\in I$, and $b=2$ in the case when
$n+j\notin I$. But $n+j\in I$ for infinitely many $n$, and also
$n+j\notin I$ for infinitely many $n$. This contradiction shows that
$\theta(0^{\infty})=0^{\infty}$.

Since $\theta(0^{\infty})=0^{\infty}$, we have $\theta(W_0)=W_0$ and
$\theta(1^{\infty})=c^{\infty}$ for some $c\in\{1,2\}$. By (b) and
(d), $0^{\infty}.01^n 0^{\infty}\in W$ for all $n\in I$, and there
are blocks $p$, $q$ and integers $l$, $m$ such that $l\geq -2N$ and
$$
\sigma_W^m \theta (0^{\infty}.01^n 0^{\infty})=0^{\infty}p.0c^{n+l}
0q0^{\infty}
$$
for all $n\in I$ with $n\geq 2N+1$. There are infinitely many $n\in
I$ such that $n+l\in I$, but if $l\ne 0$ then there are also
infinitely many $n\in I$ such that $n+l\notin I$. By the same
reasoning as in the proof of $\theta(0^{\infty})=0^{\infty}$, we
conclude that $l=0$ and $c=1$. Thus we have
\begin{equation}
\sigma_W^m \theta (0^{\infty}.01^n 0^{\infty})=0^{\infty}p.01^{n}
0q0^{\infty} \qquad (n\in I, \  n\geq 2N+1).
\end{equation}

To prove (3.1), suppose that $x\in W_0$ and $x_0 = 0$. Let $L$ be a
positive integer such that $L\geq N + |m|$ and $x_i=0$ for $|i|>L$.
Let $n\in I$ be such that $n\geq \max\{ 3L +2, |q| +1\}$. Then
$x_{[-L, -1]} 0 x_{[1, L]} 0^{n-L} 1^n$ is stable,
$$
z=0^{\infty} x_{[-L, -1]}. 0 x_{[1, L]} 0^{n-L} 1^n  0^{\infty}\in
W_0,
$$
$\sigma_W^m \theta (x)_0 = \sigma_W^m \theta (z)_0$, and (3.3)
implies that
$$
\sigma_W^m \theta (z)_{[n, 3n+1]} = 0 1^n 0 q 0^{n-|q|}.
$$
In particular, $\sigma_W^m \theta (z)_{[n, 2n+1]} = 0 1^n 0$ and
$\sigma_W^m \theta (z)_{3n+1} = 0 $. Since $n\in I$, (g) implies
that $\sigma_W^m \theta (z)_0 = 0$. Hence $\sigma_W^m \theta (x)_0 =
0$. Since $\sigma_W(W_0)=W_0$, we conclude that
$$
\mathcal Z(x) \subset \mathcal Z(\sigma_W^m \theta (x)) \qquad (x\in
W_0).
$$
If we apply this result to $\theta^{-1}$, we see that there is an
integer $m'$ such that
$$
\mathcal Z(x) \subset \mathcal Z(\sigma_W^{m'} \theta^{-1} (x))
\qquad (x\in W_0).
$$
We then have
$$
\mathcal Z(x) \subset \mathcal Z(\sigma_W^m \theta (x)) \subset
\mathcal Z(\sigma_W^{m'} \theta^{-1} \sigma_W^m \theta (x)) =
\mathcal Z(\sigma_W^{m' + m} (x)) \qquad (x\in W_0),
$$
but this implies that $m'+m=0$, and we obtain (3.1).
\end{proof}

\end{document}